\documentclass[conf]{new-aiaa} %for conference papers
\usepackage[utf8]{inputenc}

\usepackage{graphicx}

\usepackage{amsmath}
\usepackage{amsmath} % assumes amsmath package installed

\usepackage{amsfonts}
\usepackage{amsthm}

\newtheorem{proposition}{Proposition}[section]

\usepackage{float} 
\usepackage{subcaption}

\usepackage{tikz}
\usetikzlibrary{circuits.ee.IEC}
\usetikzlibrary{shapes,arrows,calc,positioning}
\tikzstyle{bigblock} = [draw, fill=blue!20, rectangle, 
    minimum height=6em, minimum width=8em]
\tikzstyle{medblock} = [draw, fill=blue!20, rectangle, 
    minimum height=4em, minimum width=4em]    
\tikzstyle{mux} = [draw, fill=black!20, rectangle, 
    minimum height=5em, minimum width=0.1em]    
\tikzstyle{smallblock} = [draw, fill=blue!20, rectangle, 
    minimum height=2em, minimum width=3em]
    
\tikzstyle{data_block} = [draw, fill=green!20, rectangle, 
    minimum height=2em, minimum width=3em]
\tikzstyle{ops_block} = [draw, fill=blue!20, rectangle, 
    minimum height=2em, minimum width=3em]    
\tikzstyle{est_block} = [draw, fill=red!20, rectangle, 
    minimum height=2em, minimum width=3em]    
    
\tikzstyle{sum} = [draw, fill=blue!20, circle, node distance=1cm,minimum height=0.5cm]
\tikzstyle{signal} = [coordinate]
\tikzstyle{pinstyle} = [pin edge={to-,thin,black}]
\tikzstyle{block} = [draw, fill=blue!20, rectangle, 
    minimum height=3em, minimum width=9em]
\tikzstyle{blockS} = [draw, fill=blue!20, rectangle, 
    minimum height=3em, minimum width=4em]    
\tikzstyle{input} = [coordinate]
\tikzstyle{output} = [coordinate]

\tikzstyle{gain} = [draw, fill=blue!20, regular polygon, regular polygon sides=3, shape border rotate=30]
\tikzstyle{gain_vert} = [draw, fill=blue!20, regular polygon, regular polygon sides=3, shape border rotate=0]

\usetikzlibrary{matrix}

\usetikzlibrary{positioning}
\usetikzlibrary{math}

\usepackage{hyperref}
\usepackage{xcolor}
\hypersetup{
    colorlinks,
    linkcolor={blue!100!black},
    citecolor={blue!50!black},
    urlcolor={blue!80!black}
}

\newcommand{\bc}{\begin{center}}
\newcommand{\ec}{\end{center}}
\newcommand{\benum}{\begin{enumerate}}
\newcommand{\eenum}{\end{enumerate}}

\newcommand{\matl}{\left[ \begin{array}}
\newcommand{\matr}{\end{array} \right]}

\renewcommand{\matl}{\begin{bmatrix}}
\renewcommand{\matr}{\end{bmatrix}}

\newcommand{\matls}{\left[ \begin{smallmatrix}}
\newcommand{\matrs}{\end{smallmatrix} \right]}
\newcommand{\isdef}{\stackrel{\triangle}{=}}

\newcommand{\inv}{^{-1}}

\newcommand{\dpder}[2]{\displaystyle\frac{\partial {#1}}{\partial {#2}}}

\newcommand{\tr}{{\rm tr}\,}
\newcommand{\rmN}{{\rm N}}

\newcommand{\rmT}{{\rm T}}

\newcommand{\rma}{{\rm a}}

\newcommand{\BBR}{{\mathbb R}}

\newcommand{\SN}{{\mathcal N}}
\newcommand{\SO}{{\mathcal O}}
\newcommand{\SP}{{\mathcal P}}

\usepackage{enumitem}
\newlist{todolist}{itemize}{2}
\setlist[todolist]{label=$\square$}
\usepackage{pifont}

\usepackage[colorinlistoftodos,textwidth=1 in,textsize=footnotesize]{todonotes}
\title{Thrust Regulation in a Solid Fuel Ramjet using \\ Dynamic Mode Adaptive Control}

\author{
Parham Oveissi\footnote{Graduate Research Assistant, Department of Mechanical Engineering, University of Maryland, Baltimore County, 1000 Hilltop Circle, Baltimore, MD 21250.}, 
Gohar T. Khokhar\footnote{Postdoctoral Research Associate, Department of Aerospace \& Mechanical Engineering, University of Arizona, 1130 N. Mountain Avenue, Tucson, AZ 85721. AIAA Member.},
Kyle Hanquist\footnote{Assistant Professor, Department of Aerospace \& Mechanical Engineering, University of Arizona, 1130 N. Mountain Avenue, Tucson, AZ 85721. AIAA Senior Member.}, 
Ankit Goel\footnote{Assistant Professor, Department of Mechanical Engineering, University of Maryland, Baltimore County, 1000 Hilltop Circle, Baltimore, MD 21250.}
}

\begin{document}
\maketitle

\begin{abstract}
    % This paper presents the application of a novel data-driven adaptive control technique, called dynamic mode adaptive control, to regulate the thrust generated by a solid fuel ramjet (SFRJ). 
    % % 
    % A high-fidelity computational model combining compressible flow theory with equilibrium chemistry is developed to simulate combustion dynamics. 
    % % 
    % The dynamic mode adaptive control is based on dynamic mode approximation and a tracking controller based on the identified approximation. 
    % % 
    % Preliminary numerical results suggest that the dynamic mode adaptive control is a reliable mechanism to regulate the SFRJ thrust.
    
    This paper presents the application of a novel data-driven adaptive control technique, called dynamic mode adaptive control (DMAC), for regulating thrust in a solid fuel ramjet (SFRJ). A high-fidelity computational model incorporating compressible flow theory and equilibrium chemistry is used to simulate the combustion dynamics. An adaptive tracking controller is designed using the DMAC framework, which leverages dynamic mode decomposition to approximate the local system behavior, followed by a tracking controller synthesized around the identified model. Simulation results demonstrate that DMAC provides an effective and reliable approach for thrust regulation in SFRJs. In addition, a systematic hyperparameter sensitivity study is conducted by varying the tuning parameters over several orders of magnitude. The resulting responses show that the closed-loop performance and tracking error remain stable across wide parameter variations, indicating that DMAC exhibits strong robustness to hyperparameter tuning.

    % Solid fuel ramjets, although simple in design and operation, are computationally expensive to model due to multiphysics processes involved including combustion, mixing, and high speed flow. 
\end{abstract}

\section{Introduction}

Ramjet engines are well-suited for long-range, high-speed missions, owing to their ability to deliver sustained thrust over extended durations. 
Their operational simplicity, due to the absence of rotating turbomachinery, makes them easier to operate and maintain compared to air-breathing propulsion systems.
Based on the type of fuel used, ramjets can be classified as either liquid-fuel (LFRJ) or solid-fuel (SFRJ) variants. 
Among these, SFRJs offer greater mechanical simplicity at comparable scales, as they do not require turbopumps, fuel bladders, injectors, or the associated plumbing.
Furthermore, the higher volumetric energy density of solid fuels can enable SFRJs to achieve longer ranges than their liquid-fueled counterparts.
An additional advantage of SFRJs is their combustion behavior: the flame front typically extends along the entire length of the fuel grain, which helps suppress combustion instabilities that are more common in LFRJs.

% The development of Solid Fuel Ramjets (SFRJs) has been an active area of research for over four decades. Since the 1980s, numerous researchers have created computational models to simulate the internal flow conditions of SFRJs, contributing significant insights into the field \cite{netzer1977modeling,stevenson1981primitive,ben1999theoretical,Sun2009_simulation,wang2015numerical}. 
% %
% These numerical simulations have underscored the intricate flow interactions and provided qualitative assessments of SFRJ behavior under steady-state conditions.
% %
% More recently, the JENRE framework has been employed to simulate flow in multi-physics scenarios \cite{schwer2018liquid,schwer2019progress,kessler2022performance}. Notably, JENRE has been utilized to model unstart phenomena in supersonic combustors \cite{goodwin2022simulating}.

% Figure \ref{fig:SFRJ_schematic} shows a typical Solid Fuel Ramjet (SFRJ) geometry. 
% The solid fuel grain lines the combustor, where the high-speed, high-temperature flow vaporizes and ignites the fuel, subsequently adding energy to the flow and generating thrust. 
%
Figure \ref{fig:SFRJ_schematic} illustrates the typical geometry of a Solid Fuel Ramjet (SFRJ). 
The solid fuel grain lines the combustor wall, where the high-speed, high-temperature airflow vaporizes and ignites the fuel.
The resulting combustion adds energy to the flow, producing thrust.
Stable operation requires carefully maintained flow conditions within the combustor. 
If the airflow rate is too low, the heat release may be insufficient to generate the required thrust.
On the other hand, an excessively high flow rate can lead to inlet unstart due to excessive heat addition or exceed the blowoff limits of the combustor \cite{arjun2022unstart,varshney2019unstart,yuceil2009active}.
Both situations increase the risk of flame extinction and thrust loss.
As such, the thermodynamic state within the combustor, comprised of pressure, temperature, and mass flow rate, must remain within a narrow, controlled operating range.
Thus, ensuring reliable operation of an SFRJ requires a control system that maintains the SFRJ's state within acceptable limits and remains robust against parametric variations and external disturbances across the broadest possible operating envelope.
However, due to the complexity of the multi-physics processes involved, including solid combustion, mixing, and high-speed flow, developing an analytical model that 1) predicts a stable operational range, 2) is computationally efficient, and 3) can be used for control system design, is highly challenging.

\begin{figure}[htbp]
    \centering
    \includegraphics[clip=true,trim=02 02 03 02,width = 0.7\textwidth]{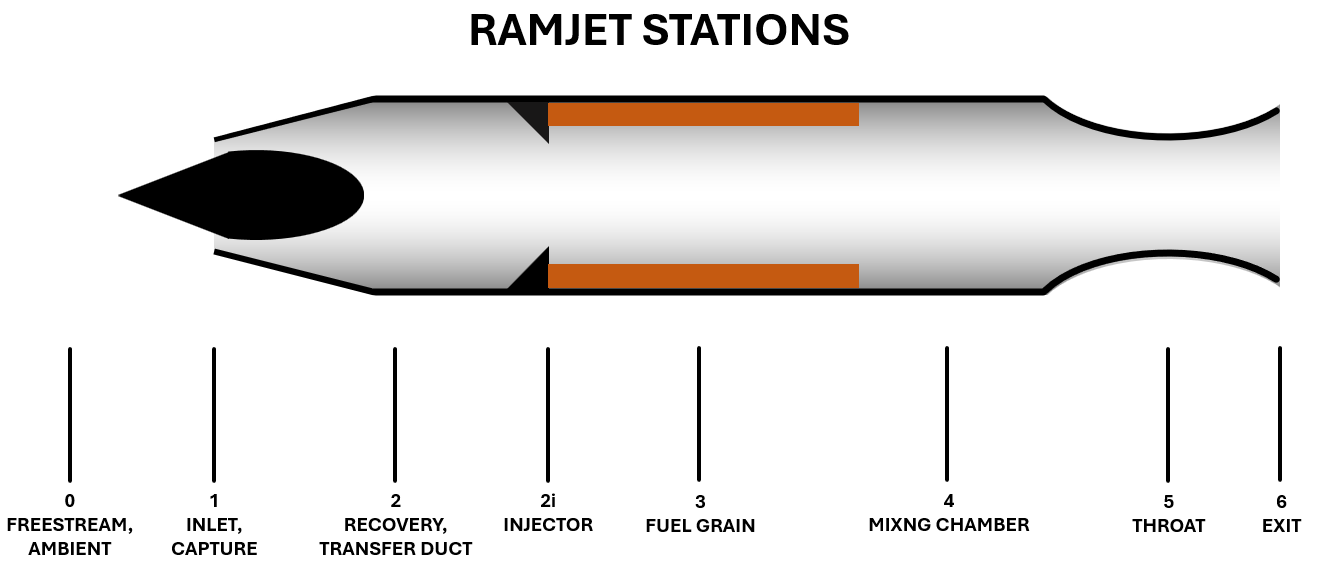}
    % \newline
    \caption{A typical SFRJ cross section.}
    \label{fig:SFRJ_schematic}
\end{figure}

In this paper, we consider the problem of developing an adaptive control system that does not require a system model, but instead uses limited measurements to update the control law and generate the required control signal to regulate the thrust generated by the SFRJ.
This work complements our previous work described in \cite{oveissi2023learning, oveissi2024adaptive,oveissi2025adaptive}, which investigated the application of the retrospective cost adaptive control to regulate the thrust generated by a quasi-static model of SFRJ and a computational model that simulates the multiphysics dynamics in an SFRJ. 

The control design approach used in this work is based on the dynamic mode adaptive control (DMAC) algorithm, which is described in detail in \cite{oveissi2025modelfreedynamicmodeadaptive}.
Although our previous work based on the retrospective cost adaptive control (RCAC) framework has been successfully applied to the thrust control problem in ramjets \cite{oveissi2025adaptive,oveissi2024adaptive,oveissi2023learning, deboskey2025situ,goel2018retrospective, oveissi2025learning}, RCAC requires the choice of a \textit{target filter} that captures the essential modeling information required to update the control law. 
However, the fixed filter choice limits the operational envelope of the system and significantly complicates the design process in a multi-input, multi-output system.
In contrast, DMAC does not require any modeling information and instead uses the measured data to identify a low-order dynamic approximation of the system along with an adaptive linear controller and thus potentially has a larger operational envelope.

To simulate the flow inside an SFRJ, we use a high-fidelity computational model described in detail in \cite{khokhar2025investigation}.
Specifically, a truncated section of an SFRJ geometry is used to simulate the internal flow in order to keep computational cost low. 
To model the continuum hypersonic flows, we use the SU2-NEMO (NonEquilibrium MOdels) code that solves the Navier-Stokes equations for multi-species gases in thermochemical nonequilibrium \cite{maier2021su2,economon:aj:2016,maier:a:2021}.

% , whose geometry is based on the center plane of the Hyshot-II scramjet engine \cite{smart2006flight}. 
The paper is organized as follows. 
Section \ref{sec:cfd_model} describes the computational model in detail and the simplified SFRJ geometry used in this work. 
Section \ref{sec:DMAC} briefly reviews the dynamic mode adaptive control framework to regulate the thrust of the SFRJ model.
Section \ref{sec:prelim_results} presents simulation results to demonstrate the application of the DMAC technique to regulate the SFRJ thrust. 
Finally, the paper concludes in Section \ref{sec:conclusions}.
% \begin{figure}[h]
%     \centering
%     \includegraphics[clip=true,trim=02 02 03 02,width = 0.75\textwidth]{Figures/Performance_1D.png}
%     \newline
%     \caption{Performance characteristics including $F_\mathrm{net}$ and $I_sp$ of the 1-D model $(a)$ as a function of bypass ratio $\beta$ and $(b)$ Mach number $(Ma)$.}
%     \label{fig:SFRJ_inlet_1D}
% \end{figure}

\section{Computational Model of SFRJ}
\label{sec:cfd_model}

This section briefly describes the computational model of the SFRJ considered in this work. 
The governing equations of the flow are described in detail in \cite{oveissi2025adaptive, khokhar2025investigation}.
The CFD software used for this work is SU2, a computational analysis and design package that has been developed to solve multiphysics analysis and optimization tasks using unstructured mesh topologies \cite{economon2016}.
SU2 employs a median-dual finite-volume approach to solve the discretized governing equations. 
% For this work, the convective fluxes are discretized using the Jameson-Schmidt-Turkel scheme (JST) \cite{liou1993}.
Further details of the governing equations and numerical schemes of SU2 can be found in Ref.~\cite{economon2016}.

%In order to achieve second-order accuracy in space, the upwind scheme of AUSM is used with MUSCL reconstruction \cite{vanleer1979}. 

The computational domain consists of a truncated SFRJ geometry that includes the inlet channel and the combustor, as shown in Figure \ref{fig:Truncated_vs_Full}.
The truncated geometry is considered to reduce the computational cost of the simulation. 
The inlet channel is 40 mm in diameter and has a length of 0.2 m. The combustor is 70 mm in diameter and has a length of 0.978 m. The total length for the simplified geometry is 1.178 m.
Since the geometry is axisymmetric about the centerline, only half of the two-dimensional cross-section at the center of the geometry is simulated. 
Figure \ref{fig:MachContourPlot} shows the baseline Mach number contours in the computational domain with a supersonic inlet velocity of  $695$ $\rm m/s$, a static inlet pressure of $100,000$ $\rm Pa$, and a static inlet temperature of $300$ $K$ and with no heat addition.

% At the inlet of the cylinder, a supersonic velocity of 695 m/s, a static pressure of 100,000 Pa, and a static temperature of 300 K is prescribed. 

% Figure \ref{fig:open_loop_thrust_heatflux_0_1_8} shows the steady-state thrust generated by the SFRJ for several constant heat flux values. 
% Note the nonlinear input-output relation between the thrust and heat flux.
\begin{figure}[htbp]
    \centering
    \includegraphics[width=0.9\columnwidth]{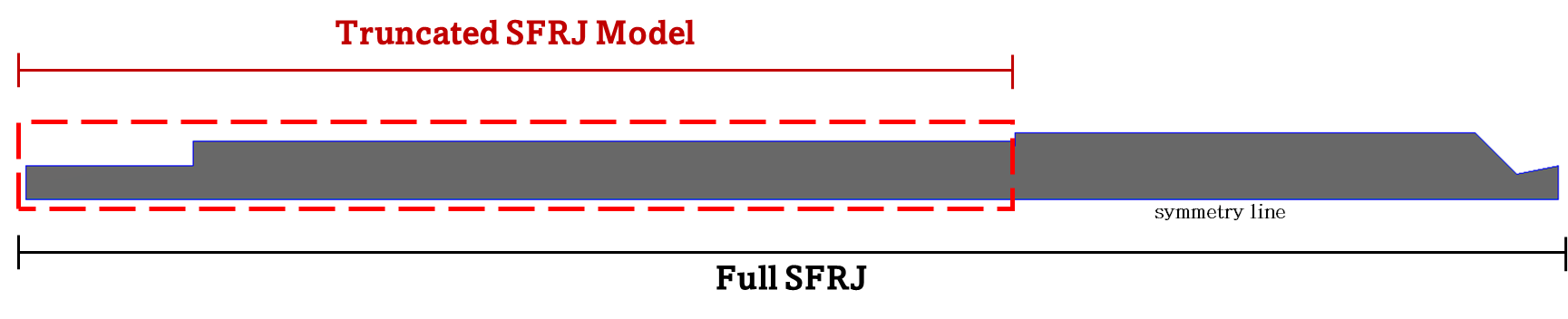}
    \caption{Truncated and the full SFRJ geometry. Only the truncated section is considered in this work.}
    \label{fig:Truncated_vs_Full}
\end{figure}

% To reduce computational resources, the present work considers a truncated of the SFRJ geometry studied in [Chinese paper], which consists of only an inlet channel and a combustor, shown in Figure \ref{fig:Truncated_vs_Full}. The inlet channel is 40 mm in diameter and has a length of 0.2 m. The combustor is 70 mm in diameter and has a length of 0.978 m. The total length for the simplified geometry is 1.178 m.

\begin{figure}[htbp]
    \centering
    \includegraphics[width=0.9\columnwidth]{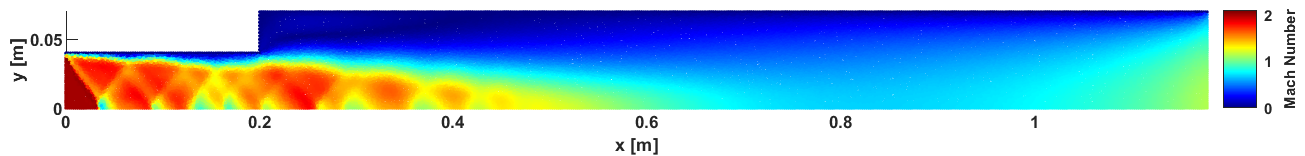}
    \caption{Mach number contour for truncated SFRJ.}
    \label{fig:MachContourPlot}
\end{figure}

% \begin{figure}[h]
%     \centering
%     \includegraphics[width=0.6\columnwidth]{Figures/open_loop_thrust_heatflux_0_1_8.eps}
%     \caption{Steady-state thrust generated by the SFRJ for several constant heat flux values.}
%     \label{fig:open_loop_thrust_heatflux_0_1_8}
% \end{figure}

\section{Dynamic Mode Adaptive Control}
\label{sec:DMAC}
This section presents the dynamic mode adaptive control (DMAC) algorithm. 
As shown in Figure \ref{fig:DMAC_architecture}, consider a dynamic system in a basic servo loop architecture whose input is $u_k \in \BBR^{l_u}$ and the output is $y_k \in \BBR^{l_y}.$
% Letting $T_\rms>0$ denote the sample time, the system's output is sampled to generate the sampled measurements $y_k \isdef y(k T_\rms).$
% The continuous-time control signal $u(t)$ is generated using zero-order hold, that is, $u(t) = u_k$ for all $t \in [kT_\rms, (k+1) T_\rms),$ where $u_k$ is the discrete-time input signal. 
% The system $\SM$ may be continuous or discrete. 
The objective of the DMAC controller is to generate a discrete-time input signal $u_k$ such that the sampled output $y_k$ tracks the reference signal $r_k.$
The DMAC controller consists of a dynamic mode approximation to approximate the local system behavior and a tracking controller designed based on the identified model. 

\begin{figure}[htbp]
    \centering
    \resizebox{0.65\columnwidth}{!}
    {
    \begin{tikzpicture}[auto, node distance=2cm,>=latex',text centered, line width = 1.5]

        \draw[draw=black, fill=yellow!10] (-2,-2.25)              
             rectangle ++(3.75,3.5) node [xshift=-5.0em, yshift=-1em] {\textbf{DMAC}} ;

        % \draw[draw=black, fill=green!10] (2.5,-1.25)              
        %      rectangle ++(4.75,2.5) node [xshift=-6.5em, yshift=-6em] {\textbf{SFRJ}} ;
             
        % \node at (-3,0) (reference) {$r_k$};
        \node [smallblock, blue, fill = blue!20, minimum width=6em, minimum height=3em] (Controller) {Control};
        % {$\begin{array}{c} \text{Control} \\ \text{law} \end{array}$};

        \node [smallblock, , fill = red!20, right = 5 em of Controller, minimum width = 6em, minimum height=3em] (plant) {System};

        % \node [smallblock, red, fill=red!20,right = 2 em of plant, minimum height=3em] (Plant) {$ \SM$};
        
        %%%%%%%%%%%
        % \node[circle,draw=black, fill=white, inner sep=0pt,minimum size=3pt] (rc11) at ([xshift=5em,yshift=1em]Plant) {};
        % \node[circle,draw=black, fill=white, inner sep=0pt,minimum size=3pt] (rc21) at ([xshift=4em,yshift=1em]Plant) {};
        % \draw [-] (rc21.north east) --node[below,yshift=.55cm]{$T_\rms$} ([xshift=.3cm,yshift=.15cm]rc21.north east) {};
        %%%%%%%%%%%

        %%%%%%%%%%%
        % \node[circle,draw=black, fill=white, inner sep=0pt,minimum size=3pt] (rc11_xi) at ([xshift=5em,yshift=-1em]Plant) {};
        % \node[circle,draw=black, fill=white, inner sep=0pt,minimum size=3pt] (rc21_xi) at ([xshift=4em,yshift=-1em]Plant) {};
        % \draw [-] (rc21_xi.north east) --
        % % node[below,yshift=.55cm]{$T_\rms$}
        % ([xshift=.3cm,yshift=.15cm]rc21_xi.north east) {};
        %%%%%%%%%%%
        
        % \node [smallblock, right = 3 em of Plant] (output_mat) {$C$};
        \node [smallblock, blue, fill = blue!20, below = 2 em of Controller, minimum width=6em] (DMA) {DMA};

        \draw[<-] (Controller.160) -- +(-2,0) node[xshift = 1em, yshift = 0.75em]{$r_k$};
        \draw[->] (plant) -- +(3,0) node[xshift = -1em, yshift = 0.75em]{$y_k$};        
        \draw[->] (Controller) node[xshift = 6em, yshift = 0.75em]{$u_k$} -- (plant);
        % \draw[->] (plant) node[xshift = 2.6em, yshift = 0.75em]{$u(t)$} -- (Plant);
        % \draw[-] (Plant.32) node[xshift = 1em, yshift = 0.75em]{$y(t)$} -- (rc21) ;
        % \draw[-] (Plant.-32) node[xshift = 1em, yshift = 0.75em]{$\xi(t)$} -- (rc21_xi) ;
        % \draw[->] (rc11) -- (output);

        \draw[->] (plant) -| +(2,-2.5) |- (-2.5,-2.5) |-(Controller.180);
        \draw[->,blue] (plant.-22) node[xshift = 1.25em, yshift = -0.75em]{$\xi_k$} -| +(0.65,-1.25) |-(DMA.-10);
        
        \draw[blue,->] (Controller.0) -| +(0.4,-1) node[xshift = -0.75em, yshift = 0.1em]{$u_k$} |- (DMA.10);
        \draw[blue,->] (DMA.180) node[xshift = 0.25em, yshift = 2em]{$A_k, B_k$} -| +(-0.5,1)  |- (Controller.200);

    \end{tikzpicture}
    }
    % \vspace{-2em}
        \caption{Dynamic Mode Adaptive Control (DMAC) architecture for model-free, data-driven, and learning-based control of dynamic systems.         
        }
        \label{fig:DMAC_architecture}
    \end{figure}
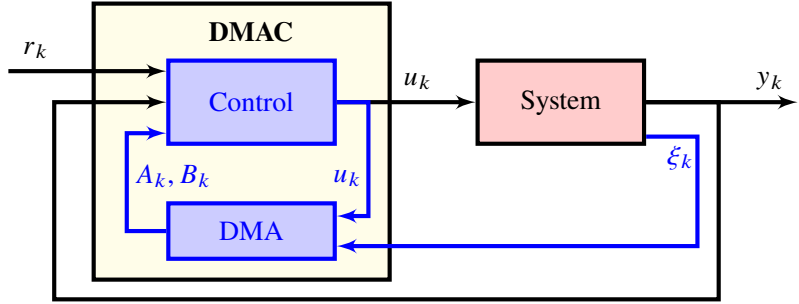

\subsection{Dynamic Mode Approximation}
\label{sec:DynApprox}
Let $\xi_k \in \BBR^{l_\xi}$ denote the measured portion of the state of the system. 
Note that $\xi_k$ may or may not be the entire state of the system.
To compute the control signal $u_k$, we first approximate linear maps $A \in \BBR^{l_\xi \times l_\xi}$ and $B \in \BBR^{l_\xi \times l_u}$ such that 
\begin{align}
    \xi_{k+1} = A \xi_{k} + B u_{k},
    \label{eq:linear_approximation}
\end{align}
which can be reformulated as
\begin{align}
    \xi_{k+1} = \Theta \phi_k.
    \label{eq:linear_approximation_AxForm_1}
\end{align}
where 
\begin{align}
    \Theta &\isdef \matl A  & B \matr \in \BBR^{l_\xi \times (l_\xi + l_u)}, 
    \quad
    \phi_k \isdef \matl \xi_{k} \\ u_{k} \matr \in \BBR^{l_\xi+l_u}. 
\end{align}
A matrix $\Theta$ such that \eqref{eq:linear_approximation_AxForm_1} is satisfied may not exist. 
However, an approximation of such a matrix can be obtained by minimizing 
\begin{align}
    J_k (\Theta)
        \isdef 
            \sum_{i=0}^{k} &\lambda^{k-i} \| \xi_{k} - \Theta \phi_{k-1} \|^2_2 
            % \neweqline
            +
            \lambda^k \tr (\Theta^\rmT R_\Theta \Theta) ,
    \label{eq:J_k_def}
\end{align}
where
% $\|M\|_\rmF \isdef \tr(M M^\rmT)$ is the Frobenious norm of the matrix $M$ \cite{strang2022introduction} and
$R_\Theta \in \BBR^{(l_\xi+l_u) \times (l_\xi+l_u)}$ is a positive definite regularization matrix that ensures the existence of the minimizer of \eqref{eq:J_k_def}
and 
$\lambda \in (0,1]$ is a forgetting factor. 
In nonlinear or time-varying systems, $\Theta$ approximated by minimizing \eqref{eq:J_k_def} in the case where the state varies significantly in the state space may not be able to capture the local linear behavior at the current state.
Thus, incorporating a geometric forgetting factor to prioritize recent data over older data improves the linear approximation and prevents the algorithm from becoming sluggish.

\begin{proposition}
    Consider the cost function \eqref{eq:J_k_def}.
    % Let $\Theta_k$ denote the minimizer of \eqref{eq:J_k_def}.
    For all $k\geq 0,$ define the minimizer of \eqref{eq:J_k_def} as
    \begin{align}
        \Theta_k 
            \isdef 
                \min_{\Theta \in \BBR^{l_x \times (l_x + l_u)}} J_k(\Theta).
    \end{align}
    Then, the minimizer $\Theta_k$ satisfies
    \begin{align}
        \Theta_k
            &=
                \Theta_{k-1} 
                +
                \left(
                    \xi_{k} - \Theta_{k-1} \phi_{k-1}
                \right)
                \phi_{k-1}^\rmT \SP_k  
            , \\
        \SP_k
            &=
                \lambda \inv \SP_{k-1} 
                -
                \lambda \inv
                \SP_{k-1} \phi_{k-1}
                \gamma_k \inv 
                \phi_{k-1}^\rmT \SP_{k-1},
    \end{align}
    where, for all $k \geq 0,$ 
    $\gamma_k \isdef \lambda  +  \phi_{k-1}^\rmT \SP_{k-1} \phi_{k-1},$ and  
    $\Theta_0 = 0,$
    $\SP_0 \isdef R_\Theta\inv. $
\end{proposition}

\begin{proof}
    See Proposition V.2 in \cite{oveissi2025modelfreedynamicmodeadaptive}.
\end{proof}

Note that the cost function \eqref{eq:J_k_def} is a matrix extension of the cost function typically considered in engineering applications \cite{goel2020recursive}.
As shown in \cite{Mareels1986,Mareels1988,goel2020recursive}, persistency of excitation is required to ensure that 1) the estimate converges and 2) the corresponding covariance matrix $\SP_k$ remains bounded. 
To ensure the persistency of excitation, in this paper, we introduce a zero-mean white noise in the control signal to promote persistency in the regressor $\phi_k$, as discussed in Section \ref{sec:controlUpdate}.

\subsection{Tracking Controller}
\label{sec:controlUpdate}
This subsection presents the algorithm to compute the control signal $u_k$ using the dynamics approximation computed in Section \ref{sec:DynApprox}.
To track the reference signal $r_k  \in \BBR$, the DMAC algorithm uses the fullstate feedback controller with integral action.
% described in Appendix \ref{sec:FSFi}.
Note that the full state refers to the state $\xi_k$ and not the system state $x_k.$ 
In particular, the control law is 
\begin{align}
    u_k = K_{\xi,k} \xi_k + K_{q,k} q_k + v_k,
\end{align}
where
\begin{align}
    \xi_k 
        \isdef
            \matl V_{\rma, k} \\ P_{\rma, k} \matr \in \BBR^2, 
\end{align}
denotes the DMAC state vector.
Here, $V_{\rma, k}$ and $P_{\rma, k}$ represent the normalized average flow velocity and average pressure at the nozzle outlet, respectively.
The normalization is performed by dividing each state value by its corresponding reference value obtained under nominal heat flux conditions of $2 \times 10^6~\rm{W/m^2}$. That is,
\begin{align}
    V_{\rma, k} = \frac{V_{\text{out},k}}{V_{\text{ref}}}, \quad 
    P_{\rma, k} = \frac{P_{\text{out},k}}{P_{\text{ref}}},
\end{align}
where $V_{\text{ref}}$ and $P_{\text{ref}}$ denote the outlet velocity and pressure measured at the nominal heat flux condition. This normalization ensures that the regressor $\phi_k$ used in the dynamics approximation remains well-scaled and numerically stable across operating points.

The matrices $K_{\xi,k} \in \BBR^{l_u \times l_\xi } $ and $K_{q,k} \in \BBR^{l_u \times l_y}$ are the time-varying fullstate feedback gain and the integrator gain, computed using the technique shown in Appendix A of \cite{oveissi2025modelfreedynamicmodeadaptive} and $v_k \sim \SN(0,\sigma_v I_{l_u})$ is a zero-mean white noise signal added to the control to promote persistency in the regressor $\phi_k$ used in the dynamic mode approximation step.
Note that the integrator state $q_k$ satisfies
\begin{align}
    q_{k+1}
        =
            q_k + z_k,
\end{align}
where $z_k \isdef r_k - y_k$ is the output error.

% In this work, we use the full-state feedback control with integral action. 
% As shown in XYZ, the computation of the integral gain $K_{q,k}$ requires the output matrix $C_k,$ which is obtained by linearizing the output map.
% In this work, we construct a neural network-based output map, that is, 
% \begin{align}
%     y_k = NN(x_k),
% \end{align}
% and thus the linearized output matrix is given by
% \begin{align}
%     C_k 
%         =
%             \dpder{NN(x)}{x} \Bigg \vert_{x = x_k}.
% \end{align}
% \todo{Fix the NN discussion}

The gain matrices $K_{\xi,k}$ and $K_{q,k}$ are computed using the well-known linear-quadratic-integral control \cite{young1972approach}, which requires the $A, B, $ and $C$ matrices of the system. 
The dynamics matrix $A$ and the input matrix $B$ are given by the dynamic mode approximation described in \ref{sec:DynApprox}.
The computation of the output matrix $C$ is described below.

In this work, the system output is modeled using a neural network.
The neural network is trained using simulation data from the SFRJ model described in Section \ref{sec:cfd_model}.
Specifically, the thrust $y \in \BBR$ is modeled as a nonlinear function of $\xi,$ that is, 
\begin{align}
    y = NN(\xi),  
    \label{eq:output_NN}
\end{align}
where 
and $NN : \BBR^2 \rightarrow \BBR$ is a feedforward neural network \cite{rozario2024matrix}.

In this work, the neural network is trained using MATLAB’s \href{https://www.mathworks.com/help/deeplearning/ref/feedforwardnet.html}{feedforwardnet} function.
The training data consists of two sets of high-fidelity simulation results, totaling $500$ samples.
Each sample consists of the pair $(V_{\rma}, P_{\rma})$ as inputs and the corresponding generated thrust $y$ as the output.
% 
% The inputs are first standardized via z-score normalization:
% \begin{align}
%     \xi_k^{\text{norm}} = \frac{\xi_k - \mu_\xi}{\sigma_\xi},
% \end{align}
% where $\mu_\xi \in \BBR^2$ and $\sigma_\xi \in \BBR^2$ are the mean and standard deviation of the training inputs, respectively.
% These statistics are stored and reused during inference.
% 
The neural network architecture used in this work, shown in Fig.~\ref{NN_diagram}, consists of two hidden layers, each with $10$ neurons and \href{https://www.mathworks.com/help/deeplearning/ref/tansig.html}{tansig} activation functions, followed by a linear output layer.
The network is trained to minimize the mean squared error (MSE) loss between the predicted and actual thrust using the Levenberg–Marquardt algorithm.
The dataset is randomly split into training (70\%), validation (15\%), and testing (15\%) sets.
Figure \ref{fig:NN_performance} shows the mean squared error (MSE) on the training, validation, and test sets over epochs. 
All three error curves decrease consistently, indicating proper learning without overfitting. 
    The final validation error closely matches the training and test errors, suggesting that the trained network generalizes well to unseen data.
% \todo{Include training and validation loss plots}

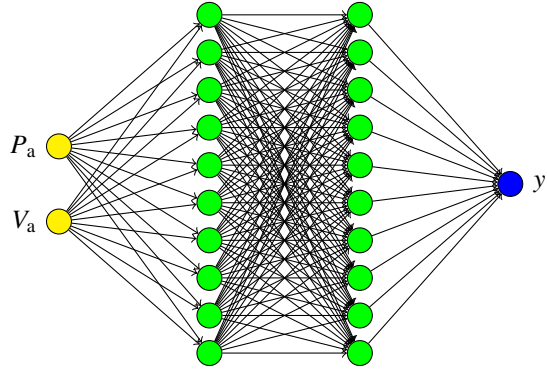
\begin{figure}[htbp]
\centering
\begin{tikzpicture}[scale=1, transform shape]

    %===================================================
    % Input Layer
    %===================================================
    \node[circle, draw, minimum size=0.3cm, fill=yellow] (I-1) at (0,0.5) {};
    \node[anchor=east] at (I-1.west) {$P_\mathrm{a}$};

    \node[circle, draw, minimum size=0.3cm, fill=yellow] (I-2) at (0,-0.5) {};
    \node[anchor=east] at (I-2.west) {$V_\mathrm{a}$};

    %===================================================
    % Hidden Layer 1 (10 neurons)
    %===================================================
    \foreach \i in {1,...,10} {
        \node[circle, draw, minimum size=0.3cm, fill=green] 
        (Hone-\i) at (2,2.75-\i*0.5) {};
    }

    % Label

    %===================================================
    % Hidden Layer 2 (10 neurons)
    %===================================================
    \foreach \i in {1,...,10} {
        \node[circle, draw, minimum size=0.3cm, fill=green] 
        (Htwo-\i) at (4,2.75-\i*0.5) {};
    }

    % Label

    %===================================================
    % Output Layer
    %===================================================
    \node[circle, draw, minimum size=0.3cm, fill=blue] (O-1) at (6,0) {};
    \node[anchor=west] at (O-1.east) {$y$};

    %===================================================
    % Connections: Input → Hidden Layer 1
    %===================================================
    \foreach \i in {1,2} {
        \foreach \j in {1,...,10} {
            \draw[->] (I-\i) -- (Hone-\j);
        }
    }

    %===================================================
    % Connections: Hidden Layer 1 → Hidden Layer 2
    %===================================================
    \foreach \i in {1,...,10} {
        \foreach \j in {1,...,10} {
            \draw[->] (Hone-\i) -- (Htwo-\j);
        }
    }

    %===================================================
    % Connections: Hidden Layer 2 → Output
    %===================================================
    \foreach \i in {1,...,10} {
        \draw[->] (Htwo-\i) -- (O-1);
    }

\end{tikzpicture}

\caption{Artificial neural network architecture used in this work, consisting of two hidden layers with 10 neurons each and \texttt{tansig} activation functions.}
\label{NN_diagram}
\end{figure}

\begin{figure}[htbp]
    \centering
    \includegraphics[width=0.7\columnwidth]{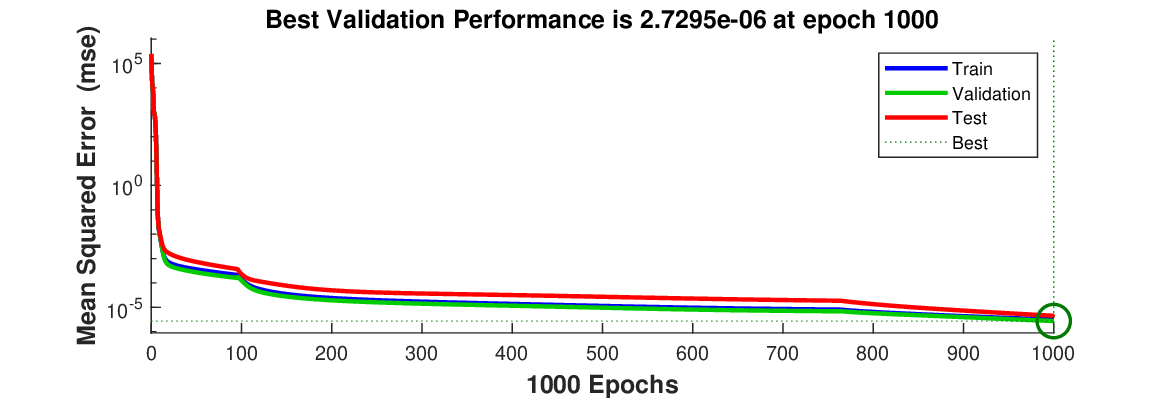}
    \caption{Neural network training performance plot showing the mean squared error (MSE) on the training, validation, and test sets over epochs. 
    } 
    \label{fig:NN_performance}
\end{figure}

It follows from \eqref{eq:output_NN} that the linearized output matrix is given by
\begin{align}
    C_k 
        =
            \dpder{NN(\xi)}{\xi} \Bigg \vert_{\xi = \xi_k},
\end{align}
where the Jacobian is numerically computed using the central difference scheme, that is, 
\begin{align}
    \frac{\partial NN(\xi)}{\partial \xi} \approx \frac{1}{2\varepsilon} 
        \begin{bmatrix}
            NN(\xi + \varepsilon e_1) - NN(\xi - \varepsilon e_1) \\
            NN(\xi + \varepsilon e_2) - NN(\xi - \varepsilon e_2)
        \end{bmatrix}.
\end{align}
Note that $e_1, e_2 \in \BBR^2$ are the standard basis vectors, and, in this work, we set $\varepsilon = 10^{-7}$.

% After training, the resulting network defines a smooth nonlinear function $NN : \BBR^2 \rightarrow \BBR$ that maps the normalized input vector $\xi_k^{\text{norm}}$ to predicted thrust $y_k$.
% To use this output model within the DMAC framework, the Jacobian of the output with respect to the original (unnormalized) inputs is computed online via central differences:

% \begin{align}
%     \frac{\partial NN(\xi_k)}{\partial \xi_k} \approx \frac{1}{2\varepsilon} 
%         \begin{bmatrix}
%             NN(\xi_k^{\text{norm}} + \varepsilon e_1) - NN(\xi_k^{\text{norm}} - \varepsilon e_1) \\
%             NN(\xi_k^{\text{norm}} + \varepsilon e_2) - NN(\xi_k^{\text{norm}} - \varepsilon e_2)
%         \end{bmatrix}^\rmT,
% \end{align}
% where $e_1, e_2 \in \BBR^2$ are standard basis vectors, and $\varepsilon = 10^{-7}$.
% To map the gradient from normalized to unnormalized coordinates, the chain rule is applied:
% \begin{align}
%     C_k = \frac{\partial NN}{\partial \xi_k} = \frac{1}{\sigma_\xi^\rmT} \circ \left. \frac{\partial NN}{\partial \xi_k^{\text{norm}}} \right|_{\xi_k},
% \end{align}
% where $\circ$ denotes element-wise division.
% The resulting Jacobian $C_k \in \BBR^{1 \times 2}$ is used as a local linear output map in the full-state feedback control law with integral action (see Appendix~\ref{sec:FSFi}).

% This approach allows us to extract a locally linear output model $y_k \approx C_k \xi_k$ from a nonlinear neural network, enabling its integration with the linear DMAC framework in real time.

\section{Simulation Results}
\label{sec:prelim_results}
This section presents numerical examples that demonstrate the application of the DMAC technique for regulating the thrust generated by the SFRJ. The DMAC hyperparameters were selected through a simple grid search to achieve a satisfactory transient response under nominal conditions.

The control signal $u_k$ is used to modulate the heat flux $w_k$ as

\begin{align}
    w_k
        =
            w_0 - K_w \times u_k,
    \label{eq:linear_map}
\end{align}
where $w_0$ is the nominal heat flux, $u_k$ is the adaptive control signal generated by the DMAC algorithm, and $K_w$ is the scaling factor.
% Note that the units of the heat flux are $\rm W/m^2.$
% 
The linear map \eqref{eq:linear_map} is chosen such that the magnitude of the adaptive control signal $u_k$ remains close to $\SO(1)$ to ensure the numerical stability of the DMAC algorithm.
In this work, the nominal heat flux $w_0 = 2 \times 10^6$ $\rm W/m^2$ and the scaling factor $K_w$ is set to $10^5.$

\subsection{Command Following}
First, we consider the problem of regulating the SFRJ thrust to a constant value. 
In particular, the SFRJ is commanded to generate a constant thrust value of $r = 1000 $ $\rmN$. 
% 
% To apply DMAC, we assume that 
% \begin{align}
%     \xi_k
%         =
%             \matl
%                 q(k T_\rms) \\
%                 \dot q(k T_\rms)
%             \matr.
% \end{align}
% \todo{Fix the state and output definition}
% Note that $y_k \isdef y(T_\rms k) = q(k T_\rms).$
% To apply DMAC and compute the control signal $u_k,$ the state is sampled at a constant time step of $T_\rms=0.1$ seconds, that is, $x_k = x(T_\rms k).$ 
Since $\xi_k \in \BBR^2$ and $u_k \in \BBR,$ it follows that $\Theta_k$ is a $2 \times 3$ matrix. 
In DMAC, we set $R_\Theta = 10^2 I_3$ and the forgetting factor $\lambda = 0.995$.
The LQR weights in the tracking controller are $R_1 = I_3$  and $R_2 = 1.$

Figure \ref{fig:DMAC_SFRJ_steps_250_lambda_0_995_R0_1e2_Q_1_R_1_sysdim_2_ref_1000_noise_1e_neg2_NNthrust_JacobianC_v3} shows the closed-loop response of the SFRJ with the DMAC algorithm updating the controller, where
a) shows the commanded thrust $r=1000$ $\rmN$ and the generated signal $y_k,$ b) shows the control signal $u_k,$ c) shows the absolute value of the tracking error $z_k \isdef y_k - r$ on a logarithmic scale, and d) shows the estimate matrix $\Theta_k$ computed by DMAC. Note that the output error approached zero.

% \begin{figure}[H]
%     \centering
%     \includegraphics[width=0.5\columnwidth]{Figures/DMAC_SFRJ_steps_250_lambda_0_995_R0_1e2_Q_1_R_1_sysdim_2_ref_1000_noise_1e_neg2_NNthrust_JacobianC.eps}
%     \caption{Closed-loop response of SFRJ with DMAC. a) shows the output $y_k$ and the reference signal $r,$ b) shows the control signal $u_k,$ and c) shows the absolute value of the tracking error $z_k$ on a logarithmic scale.} 
%     \label{fig:DMAC_SFRJ_steps_250_lambda_0_995_R0_1e2_Q_1_R_1_sysdim_2_ref_1000_noise_1e_neg2_NNthrust_JacobianC}
% \end{figure}
% \todo{Make this figure in the format we used in DMAC paper}

\begin{figure}[htbp]
    \centering
    \includegraphics[width=0.7\columnwidth]{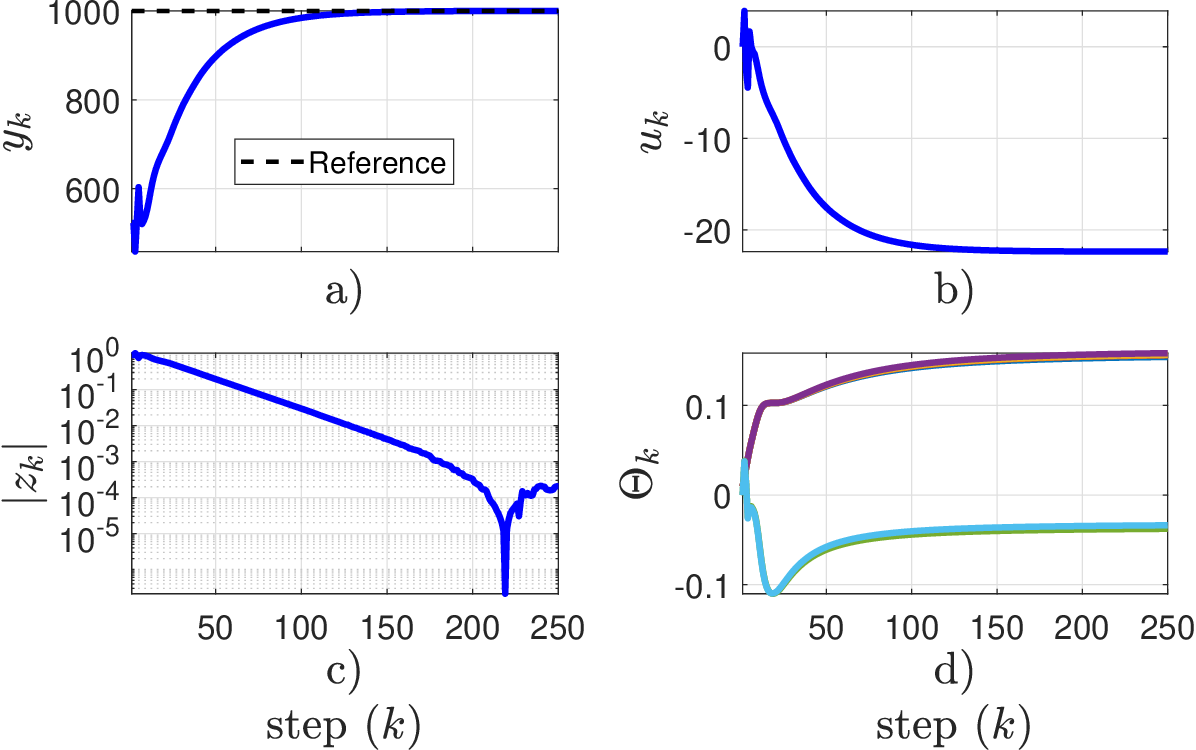}
    \caption{Closed-loop response of SFRJ with DMAC. a) shows the output $y_k$ and the reference signal $r,$ b) shows the control signal $u_k,$ and c) shows the absolute value of the tracking error $z_k$ on a logarithmic scale.} 
    \label{fig:DMAC_SFRJ_steps_250_lambda_0_995_R0_1e2_Q_1_R_1_sysdim_2_ref_1000_noise_1e_neg2_NNthrust_JacobianC_v3}
\end{figure}

Next, the SFRJ is commanded to follow a double-step command. 
Specifically, the thrust command is $r = 1000 $ $\rmN$ for $k \in (0,200)$ and $r = 1200$ $\rmN$ for $k \geq 200.$
Note that DMAC hyperparameters are kept the same as in the previous case.  
Figure \ref{fig:DMAC_SFRJ_steps_400_lambda_0_995_R0_1e2_Q_1_R_1_sysdim_2_ref_1000_1200_noise_1e_neg2_NNthrust_JacobianC_v2} shows the closed-loop response of the SFRJ with the DMAC algorithm updating the controller, where
a) shows the commanded thrust $r$ and the generated signal $y_k,$ b) shows the control signal $u_k,$ c) shows the absolute value of the tracking error $z_k \isdef y_k - r$ on a logarithmic scale, and d) shows the estimate matrix $\Theta_k$ computed by DMAC. Note that the output error approached zero.

\begin{figure}[htbp]
    \centering
    \includegraphics[width=0.7\columnwidth]{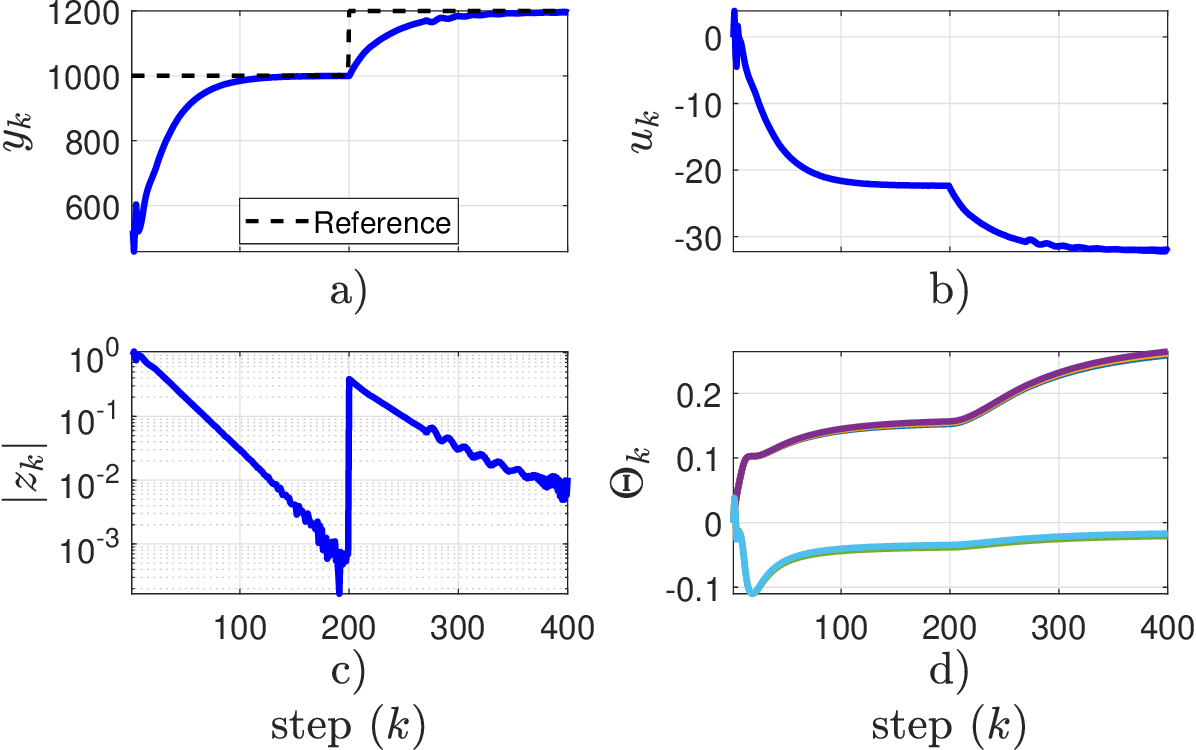}
    \caption{Closed-loop response of SFRJ with DMAC. a) shows the output $y_k$ and the reference signal $r,$ b) shows the control signal $u_k,$ and c) shows the absolute value of the tracking error $z_k$ on a logarithmic scale.} 
    \label{fig:DMAC_SFRJ_steps_400_lambda_0_995_R0_1e2_Q_1_R_1_sysdim_2_ref_1000_1200_noise_1e_neg2_NNthrust_JacobianC_v2}
\end{figure}

\subsection{Effect of Hyperparameters}
To investigate the robustness of the DMAC algorithm to its tuning hyperparameters $R_\Theta, \lambda, R_1$ and $R_2,$  we vary each of the hyperparameters systematically by keeping other hyperparameters at their nominal values.
Figure~\ref{fig:DMAC_SFRJ_HyperParam_Sensitivity}
shows the effect of the DMAC hyperparameters on the closed-loop response $y_k$ and the absolute tracking error $|z_k|$.
From top to bottom, the rows correspond to variations in $R_\Theta$, $\lambda$, $R_1$, and $R_2$, respectively. Note that, in each case, the hyperparameter is varied by a few orders of magnitude, suggesting that DMAC is robust to tuning parameters.

\begin{figure}[htbp]
    \centering
    \includegraphics[width=0.7\columnwidth]{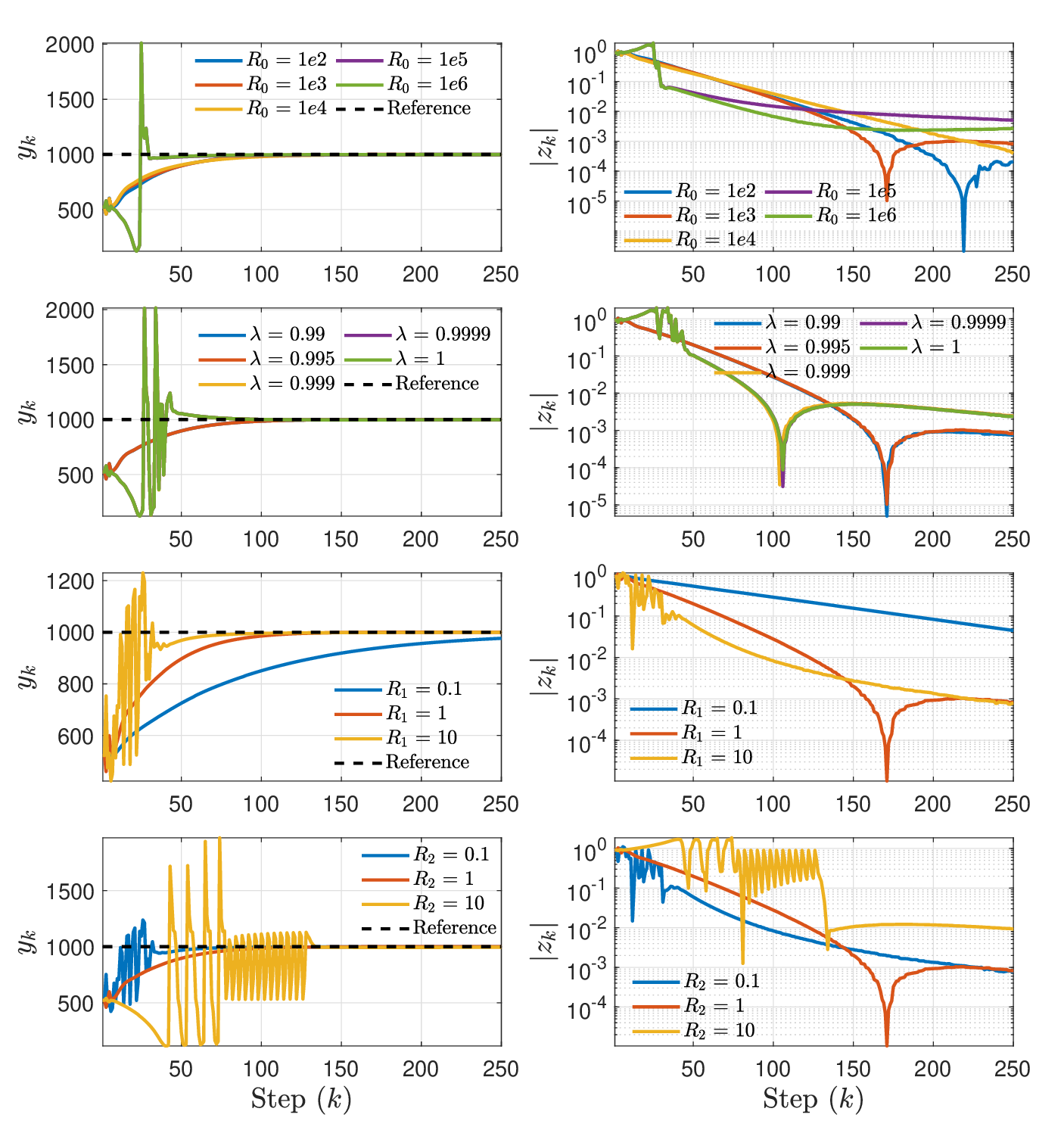}
    \caption{Effect of DMAC hyperparameters on the closed-loop performance.} 
    \label{fig:DMAC_SFRJ_HyperParam_Sensitivity}
\end{figure}

\section{Conclusion}
\label{sec:conclusions}

This paper considered the problem of regulating the thrust generated by a solid fuel ramjet (SFRJ) using limited in-situ measurements, without requiring an analytical model of the system. The simulation results presented in this work demonstrate that the dynamic mode adaptive control (DMAC) method is a viable and effective technique for thrust regulation in SFRJs. In particular, its minimal measurement requirements enhance its practicality for real-world implementation.

In addition, a systematic hyperparameter sensitivity analysis was conducted to evaluate the robustness of the DMAC framework with respect to its tuning parameters. The results indicate that the closed-loop performance and tracking error remain stable across wide variations in these parameters, suggesting that DMAC exhibits strong robustness to hyperparameter tuning. This property further supports its suitability for deployment in uncertain and dynamically changing propulsion environments.

% Furthermore, the robustness of the DMAC algorithm will be investigated by systematically varying the inlet operating conditions. 

\section{Acknowledgment}
This research was supported by the Office of Naval Research grant N00014-23-1-2468.

\bibliography{Paperbib}

\end{document}